\documentclass[11pt,reqno]{amsart}
\usepackage{amsmath,amssymb,amsthm,upref,graphicx,mathrsfs}
\usepackage{color,xcolor}
\usepackage[
  colorlinks=true,
  linkcolor=blue,
  citecolor=blue,
  urlcolor=blue]{hyperref}

\newtheorem{thm}{Theorem}[section]

\newtheorem{prop}[thm]{Proposition}

\newtheorem{rem}[thm]{\bf{Remark}}

\numberwithin{equation}{section}

\begin{document}

%--------------------------------------------------
%%Don not change any thing in this part
\leftline{ \scriptsize}

\vspace{1.3 cm}
%----------------------------------------------------------------------------
\title
{Matrices over commutative rings as sum of higher powers }
\author{Kunlathida Muangma$^A$ and Kijti Rodtes $^{B,C}$ }
\thanks{{\scriptsize
		$^{A}$ Department of Mathematics, Demonstration School, University of Phayao, Phayao 56000, Thailand, (kunlathida.ch@up.ac.th) \\
			$^{B}$ Department of Mathematics, Faculty of Science, Naresuan University, Phitsanulok 65000, Thailand, (kijtir@nu.ac.th) \\
			$^{C}$ Center of Excellence in Nonlinear Analysis and Optimization, Faculty of Science, Naresuan University, Phitsanulok 65000, Thailand, (kijtir@nu.ac.th). 		 
		\\}}
\hskip -0.4 true cm

\maketitle

%-----------------------------------------------------------------------------

%----------------------------------------------
\begin{abstract}On the Waring's problems for matrices over a commutative ring, there are some trace conditions provided for matrices eligibly expressed as a sum of $k$-th powers with $k=2,3,4,5,6,7,8$ in several literatures.   In this paper, we provide the similar conditions for matrices written as a sum of $k$-th powers with $k=9,10,11,12,13,14,15,16$.  
\end{abstract}

\vskip 0.2 true cm

%-----------------------------------------------------------------------------

\pagestyle{myheadings}
\markboth{\rightline {\scriptsize Kunlathida Muangma and Kijti Rodtes}}
{\leftline{\scriptsize }}
\bigskip
\bigskip

%-----------------------------------------------------------------------------
%-----------------------------------------------------------------------------

\vskip 0.4 true cm

\section{Introduction}
It is well known in number theory  that  every natural number can be represented as the sum of four integer squares: \lq\lq\textit{Lagrange's four-square theorem}" or \lq\lq\textit{Bachet's conjecture}".  There is also a corresponding problem asking whether each natural number $k$ has an associated positive integer $g(k)$ (depending only $k$) such that every natural number is the sum of at most $g(k)$ natural numbers raised to the power $k$, which is known as \lq\lq \textit{Waring's problems}".   This problem can be generalized naturally to deal with the rings of matrices; \lq\lq \textit{Waring's problem for matrices over rings}" (see, e.g., \cite{KG}) :
\begin{enumerate}
	\item Given a matrix $M\in M_n(R)$, can $M$ be written as a sum of $k$-th powers in $M_n(R)$?
	\item Given a ring $R$, can every $M\in M_n(R)$ be written as sum of $k$-th powers in $M_n(R)$?
	\item Find the least positive integer $g(n,k,R,M)$ such that a given matrix $M\in M_n(R)$ can be written as a sum of  $g(n,k,R,M)$ $k$-th powers in $M_n(R)$.
	\item Find the least positive integer $g(n,k,R)$ such that every matrix $M\in M_n(R)$ which can be written as a sum of  $g(n,k,R)$ $k$-th powers in $M_n(R)$.
\end{enumerate}
Here, and throughout this paper, the ring $R$ will always mean an associative ring with identity and $M_n(R)$ denotes the set of all square matrices of size $n$ over the ring $R$.
 
Some answers to this questions are known for $k=2$-nd powers and several types of rings:  for example, when $R$ is a ring satisfying $1/2\in R$ (\cite{Griffin and Krusemeyer}), when $R$ is a field of characteristic different from two and $n$ is even (\cite{Griffin and Krusemeyer}), when $R$ is a field of characteristic different from two and $M$ is not a scalar multiple of the identity (\cite{Richman2}), when $R=\mathbb{Z}$ is the ring of integers (\cite{Newman}).

For higher $k$-th powers, Richman provided a trace condition for a square matrix of size $n$ over a commutative ring $R$ to be a sum of $k$-th powers in \cite{Richman}, by using an interesting observation on the characteristic polynomials of some matrices under the restriction that $n\geq k$.  The generalised trace conditions of the question (1) and (2) for matrices over a commutative ring  that relax $n\geq k$ but instead depends on the trace of $k$-th power matrices, are obtained by  Katre and Garge in \cite{KG}.   Trace conditions of the question (1) and (2) for matrices over non-commutative rings  are also recently studied in \cite{KD} and  \cite{KW}.

For an order of algebraic number fields, necessary and sufficient conditions for the questions (1) and (2) in terms of traces and the conditions in term of discriminant based on the question (2)  are provided: for the case $k=2$ in \cite{VLN}, cases $k=3,4$ in \cite{KG}, cases $k=5,7$ in \cite{Garge} and cases $k=6,8$ in \cite{BG}.   For general commutative ring $R$, by using the generalised trace conditions, the conditions for the questions (1) and (2) can be expressed in term of traces of matrices (not the trace of power matrices) and some properties of rings: when $k=2$ in in \cite{VLN}, when $k=3,4$  in \cite{KG}, when $k=5,7$ in \cite{Garge} and recently when $k=6,8$ in \cite{BG}.

It is simple to calculate the trace of a matrix but not for the trace of a power matrix.  We observe that the existing conditions for matrices as sums of $k$-th power matrices with $k=2,3,4,5,6,7,8$ based on the calculation of traces of $k$-th power matrices and reducing them to a required form.   In fact, the calculation deals with a reduction of an additive group generated by $\{p(f(x),g(x))\;|\; f(x),g(x)\in R[x]\}$ to the additive group generated by $\{rx^s\;|\; rs=k; r,s \in \mathbb{N},x\in R\}$, where $p(x,y)$ is a fixed polynomial in $R[x,y]$ of degree $k$ (see Theorem \ref{thmtrexpansion}).

Motivated by these results, in this article, we further investigate the trace conditions for matrices over commutative rings in questions (1) and (2)  for higher $k$-th powers; precisely for $k=9,10,11,12,13,14,15,16$.   To state the first result, we recall a notion from the theory of Witt vector.   For a prime $p$ and a positive integer $s$, denote 
$$W(p,s,R):=\{a_0^{p^s}+pa_1^{p^{s-1}}+p^2a_2^{p^{s-2}}+\cdots+p^sa_s\mid a_0,a_1,a_2,\dots,a_s\in R\}.  $$
We observe that the sets $W(p,s,R)$ satisfy $W(p,s+1,R)\subseteq W(p,s,R)$ for each $s\geq 1$; precisely, for each $x=a_0 ^{p^{s+1}}+p a_1^{p^s}+p^2 a_2^{p^{s-1}}+\cdots+p^s a_s^{p}+p^{s+1}a_{s+1}\in W(p,s+1,R)$, by setting $b_i:=a_i^p$ for $i=0,1,\dots,s-1$ and $b_s:=a_s^p+pa_{s+1}$, it follows that $b_i$'s belong to $R$ and $$x=b_0 ^{p^{s}}+p b_1^{p^{s-1}}+\cdots+p^s b_s\in W(p,s,R).$$ 
Also, throughout this article, $tr(M)$ means the trace of the matrix $M$.
\begin{thm}\label{deg9}
	Let $n$ be a positive integer and $R$ a commutative associative ring with unity.  The following statements (1),(2) and (3) are equivalent:
	\begin{enumerate}
		\item Every matrix in $M_n(R)$ is a sum of $3$-rd powers of matrices in $M_n(R)$.
		\item Every matrix in $M_n(R)$ is a sum of $9$-th powers of matrices in $M_n(R)$
		\item Every element of $R$ is a $3$-rd power (mod 3R).
	\end{enumerate}
	In particular, if $R$ is an order in an algebraic number field, then these conditions are further equivalent to $(3,\operatorname{disc} R)=1$.	
	Moreover, 
	\begin{itemize}
		\item[(4)] $tr(M)\in \{a_0^9+3a_1^3\; (\operatorname{mod} 9R )\mid a_0, a_1\in R\}$ if and only if  $M$ is a sum of ninth powers in $M_n(R)$;
		\item[(5)] if $tr(M)\in W(3,s,R)$, then $M$ is a sum of ninth powers in $M_n(R)$.  
	\end{itemize}
\end{thm}
To state the following results, we denote:
\begin{eqnarray*}
S_{10}&:=&\{x_0^{10}-2x_1^5+5x_2^2 \mod 10 R\;|\; x_0,x_1,x_2\in R\}, \\
S_{12}&:=&\{x_0^{12}+2x_1^6-3x_2^4-4x_3^3+6x_4^2 \mod 12 R\;|\; x_0,x_1,x_2,x_3,x_4\in R\},\\
S_{14}&:=&\{x_0^{14}-2x_1^7+7x_2^2 \mod 14 R\;|\; x_0,x_1,x_2\in R\},\\
S_{15}&:=& \{x_0^{15}-3x_1^5+5x_2^3 \mod 15 R\;|\; x_0,x_1,x_2\in R\},
\end{eqnarray*} 
and $S^*_{12}:= \bar{S}_{12}\cup S'_{12}$, where
$$\bar{S}_{12}:=\{x_0^{12}+2x_1^6+3x_2^4+8x_3^3+12x_4^2+24x_5\;|\; x_0,x_1,x_2,x_3,x_4,x_5\in R\}$$ and 
$$S'_{12}:=\{x_0^{12}+2x_1^6+3x_2^4+8x_3^3+4x_4^{2m+1}+6x_4^2+12x_4\;|\; x_0,x_1,x_2,x_3,x_4\in R, m\in \mathbb{N}\}.$$

\begin{thm}\label{maindegcomposite}
	Let $n$ be a positive integer and $R$ a commutative associative ring with identity. Then, the followings holds
	\begin{enumerate}
		\item For each $k=10,12, 14,15$, $S_k$ is an additive group.
		\item For each $k=10,14, 15$,  $M\in M_2(R)$ is a sum of $k$-th powers of matrices in $M_2(R)$ if and only if $tr(M) \mod kR\in S_{k}$.
		\item For each $k=10,14, 15$  and $M\in M_n(R)$, if $tr(M) \mod kR \in S_{k}$, then $M$  is a sum of $k$-th powers of matrices in $M_n(R)$
		\item If $M\in M_2(R)$ is a sum of $12$-th powers of matrices in $M_2(R)$, then $tr(M) \mod 12R\in S_{12}$.
		\item For $M\in M_n(R)$, if $tr(M) \in S^*_{12}$, then $M$  is a sum of $12$-th powers of matrices in $M_n(R)$
	\end{enumerate}
\end{thm}
For prime power $p=11$ or $p=13$, we have that:
\begin{thm}\label{degprime}
	Let $n$ be a positive integers and $p=11$ or $p= 13$.  Let $R$ be a commutative associative ring with unity.  The statements (1) and (2) are equivalent:
	\begin{enumerate}
		\item Every matrix in $M_n(R)$ is a sum of $p$-th powers of matrices in $M_n(R)$.
		\item Every element of $R$ is a $p$-th power (mod pR).
	\end{enumerate}
	In particular, if $R$ is an order in an algebraic number field, then these conditions are further equivalent to $(p,\operatorname{disc} R)=1$.	
	Also, $tr(M)\in \{a_0^p+(\operatorname{mod} pR )\mid a_0\in R\}$ if and only if  $M$ is a sum of $p$ powers in $M_n(R)$. 
\end{thm}

In the following theorem, denote $W^*(2,4,R):=\bar{W}(2,4,R)\cup W(2,5,R)$, where $$\bar{W}(2,4,R):=\{a_0^{16}+2a_1^8+4a^4+16a+16a^{2m+1}\;|\; a_0,a_1,a\in R, m\in \mathbb{N}\}.$$ 
\begin{thm}\label{deg16}
	Let $n$ be a positive integer and $R$ a commutative associative ring with unity.  The following are equivalent:
	\begin{enumerate}
		\item Every matrix in $M_n(R)$ is a sum of $2$-nd powers of matrices in $M_n(R)$.
		\item Every matrix in $M_n(R)$ is a sum of $4$-th powers of matrices in $M_n(R)$
		\item Every matrix in $M_n(R)$ is a sum of $8$-th powers of matrices in $M_n(R)$.
		\item Every matrix in $M_n(R)$ is a sum of $16$-th powers of matrices in $M_n(R)$
		\item Every element of $R$ is a $2$-nd power (mod 2R).
	\end{enumerate}
In particular, if $R$ is an order in an algebraic number field, then these conditions are further equivalent to $(2,\operatorname{disc} R)=1$.
Moreover, 
\begin{itemize}
\item[(6)]  if $tr(M)\in W^*(2,4,R)$, then $M$ is a sum of sixteenth powers in $M_n(R)$. 
\end{itemize}
\end{thm}

\section{Background results}
 Trace condition for an $n\times n$ matrix over a ring $R$ to be a sum of $k$-th powers (question (1)), under the restriction that $n\geq k\geq 2$, is given by Richman in \cite{Richman}: 
 \begin{thm}[\cite{Richman}, Proposition 4.2]
 	Let $R$ be a commutative associative ring with identity.  If $n\geq k\geq 2$ are integers, then the following statement are equivalent:
 	\begin{enumerate}
 		\item $M$ is a sum of $k$-th powers in $M_n(R)$;
 		\item $M$ is a sum of seven $k$-th powers in $M_n(R)$;
 		\item $M$ belongs to $M_n(R)$ and for every prime power $p^e$ dividing $k$, there are elements $x_0=x_0(p),\dots,x_e=x_e(p)$ in $R$ such that
 		$$ tr(M)=x_0^{p^e}+px_1^{p^{e-1}}+\cdots+p^ex_e.$$
 	\end{enumerate}
 	Moreover if $k=p$ is a prime, in condition (2) \lq\lq seven" can be replaced by \lq\lq five".  Also condition (3) simplifies to \lq\lq $tr(M)=x_0^p+px_1$" for some $x_0,x_1\in R$".
 \end{thm} 
 
This result was extended by Katre and Garge in \cite{KG} by removing the restriction \lq\lq $n\geq k$":
\begin{thm}[\cite{KG}, Theorem 3.1] \label{basicthm3.1}
	Let $R$ be a commutative associative ring with identity and $n, k\geq 2$ be integers.  Let $M$ belong to $M_n(R)$.  Then the following statements are equivalent:
	\begin{enumerate}
		\item $M$ is a sum of $k$-th powers in $M_n(R)$.
		\item $tr(M)$ is a sum of traces of $k$-th powers of matrices in $M_n(R)$.
		\item $tr(M)$ is in the subgroup of $R$ generated by the traces of $k$-th powers of matrices in $M_n(R)$.
		\item $tr(M)$ is in the subgroup of $R$ generated by the traces of $k$-th powers of matrices in $M_n(R)$(mod $k!R$).
		\item $tr(M)$ is a sum of traces of $k$-th powers of matrices in $M_n(R)$(mod $k!R$).
	\end{enumerate}
\end{thm} 
This results also imply the following criterion:
\begin{thm}[\cite{KG}, Theorem 3.2]
	Let $R$ be a commutative associative ring with identity and $n, k\geq 2$ be integers.  Then the following statement are equivalent:
	\begin{enumerate}
		\item Every matrix $M$ in $M_n(R)$ is a sum of $k$-th powers in $M_n(R)$.
		\item Every element of $R$ is a sum of traces of $k$-th powers of matrices in $M_n(R)$.
		\item $R$ is generated as a group by the traces of $k$-th powers of matrices in $M_n(R)$.
		\item $R$ is generated as a group by the traces of $k$-th powers of matrices in $M_n(R)$(mod $k!R$).
		\item Every element of $R$ is a sum of traces of $k$-th powers of matrices in $M_n(R)$(mod $k!R$).
	\end{enumerate}
\end{thm} 
Katre and Garge also applied this results to answer the question (1), (2) for the case $k=3,4$.  In order to to that, they provided a useful deduction to work on $2\times 2$ matrices over $R$.
\begin{thm}[\cite{KG}, Theorem 3.5]\label{bthm35}
	Let $R$ be a commutative associative ring with identity and $n\geq m\geq 1$ and $k\geq 2$ be integers.  If every matrix in $M_m(R)$ is a sum of $k$-th powers in $M_m(R)$, then every matrix in $M_n(R)$ is a sum of $k$-th powers in $M_n(R)$.
\end{thm}

A necessary condition for the question (2) with $k=p$ is prime is also recorded in \cite{KG}.

\begin{thm}[\cite{KG}, Lemma 3.9]\label{lem39}
	Let $R$ be a commutative associative ring with identity and $n,p\geq 2$ be integers with prime $p$.  Then $tr(M^p)\equiv (tr(M))^p (mod\; pR)$ for all $M\in M_n(R)$.  Also if every $n\times n$ matrices over $R$ is a sum of $p$-th powers of matrices over $R$, then every element of $R$ is a $p$-th power (mod pR).
\end{thm}

This necessary condition can be extended to conclude that:
\begin{prop}\label{prop1}
	Let $R$ be a commutative associative ring with identity and $n,p \geq 2$ be integers with prime $p$.  Then, for any positive integers $s,t$, if every $n\times n$ matrices over $R$ is a sum of $p^t$-th powers of matrices over $R$, then every element of $R$ is a $p^s$-th power (mod pR).
\end{prop}
\begin{proof}
	Let $M\in M_n(R)$.  Suppose that there are $M_1, M_2, \dots, M_l\in M_n(R)$ such that $M=\sum_{i=1}^l M_i^{p^t}$. Then $M=\sum_{i=1}^l (M_i^{p^{t-1}})^p$  is a sum of $p$-th powers of matrices over $R$.  By Theorem \ref{lem39}, every element of $R$ is a $p$-th power (mod pR). Now, for each $\alpha\in R$, there must exist $a_1, b_1\in R$ such that $\alpha=a_1^p+pb_1$.  Since $a_1\in R$, there must exist $a_2, b_2\in R$ such that $a_1=a_2^p+pb_2$ and thus $$\alpha=(a_2^p+pb_2)^p+pb_1=a_2^{p^2}+p(b_1+\sum_{i=1}^p{p\choose i} a_2^{p-i}p^{i-1} b_2^i)=a_2^{p^2}+pb_2',$$
	where $a_2, b_2':=b_1+\sum_{i=1}^p{p\choose i} a_2^{p-i}p^{i-1} b_2^i \in R$. Continuing this process $s$-times, we conclude that every element of $R$ is a $p^s$-th power (mod pR). 
\end{proof}
In fact, by using the fact from the Witt vectors theory that $W(p,s,R)$ is closed under addition and subtraction, Richman showed that:
\begin{thm}[\cite{Richman}, Proposition 3.2]\label{prop32}
	Let $R$ be a commutative associative ring with identity and $s\geq 1$ be an integer with prime $p$. For any matrix $M\in M_n(R)$, $tr(M^{p^s})$ lies in $W(p,s,R)$.
\end{thm}

Another main tool for the trace condition based on question (1) and (2) for the case $k=2,3,4,5,6,7,8$ is the trace formula for power matrices in term of their determinant and trace.  According to Theorem \ref{lem39}, it suffices to work on the trace formula for matrices of size $2\times 2$ over $R$.  Entries of $A^n$, for $A\in M_2(R)$, in term of trace and determinant was originally given in recursive form in \cite{Mc} and its trace has been expressed explicitly in \cite{BG}.
\begin{thm}[\cite{BG}, Theorem 2.2]\label{thmtrexpansion}
	Let $R$ be a commutative associative ring with identity and $n$ be a positive integer. Let  $A\in M_2(R)$ with $t:=tr(A)$ and $\delta:=\det(A)$.   Then 
	\begin{equation}\label{trf}
	tr(A^n)=t^n+\sum_{r=1}^{[n/2]}(-1)^r\frac{n}{r}{n-r-1\choose r-1}t^{n-2r}\delta^r,
	\end{equation}
	where $[n/2]$ denotes the floor function of $n/2$.
\end{thm}

There are also necessary and sufficient conditions (for question (2)) in particular commutative rings; orders of an algebraic number field.  Recall that an order $R$ in an algebraic number field $K$ is a subring of $K$ containing the identity and also a $\mathbb{Z}$-module of $K$ with maximum rank (equal to $\deg(K/\mathbb{Q})$).  A useful criterion for every element of an order $R$ to be a prime power (modulo $pR$) in term of discriminant is given by Katre and Garge:
\begin{thm}[\cite{KG}, Lemma 4.4 and \cite{Garge}, Theorem 3.2]\label{odercri}
Let $R$ be an order in an algebraic number filed $K$ and $p$ be a prime.  The following statements are equivalent:
\begin{enumerate}
	\item Every element of $R$ is a $p$-th power modulo $pR$.
	\item $x\in R$, $x^p\in pR$ implies $x\in pR$.
	\item $(p,\operatorname{disc} R)=1$.
\end{enumerate}
\end{thm}
They used this criterion to obtain a discriminant criteria for the question (2) with the $3$ and $4$ powers in \cite{BG} and the $5$ and $7$ powers in \cite{Garge}.

\section{Proofs of the main results}

\subsection{Proof of Theorem \ref{deg9}}  Due to Theorem \ref{lem39}, each of the statements (1), (2) implies (3).  It is also obvious that (2) $\longrightarrow$ (1). So, it remains to show only that (3) $\longrightarrow$ (2) for $M_2(R)$ (by Theorem \ref{bthm35}). Suppose that (3) holds.  Let $M\in M_2(R)$.  To show that (2) holds, we use the equivalent statements of (1) and (3) in Theorem \ref{basicthm3.1}.  Let $$S^9:=\left< tr(A^{9})\mid A\in M_2(R) \right> \leq R$$ 
be the (additive) subgroup of $R$ generated by traces of $9$-th powers of matrices in $M_2(R)$.  We need to show that $tr(M):=\alpha_0\in S$.  By the assumption (3) and the same proof of Proposition \ref{prop1}, for any positive integer $m$, there exist $\alpha_1, \alpha_2,\dots,\alpha_m,\alpha_m' \in R$ such that $$\alpha_0=\alpha_1^{9}+3\alpha_2^{9}+\cdots +3^{m-1}\alpha_m^{9}+3^m \alpha_m'.  $$
Note that $$\alpha_i^9=tr(\begin{pmatrix}
\alpha_i & 0 \\
0 & 0 
\end{pmatrix}^{9})\in S^9$$ for all $i=1,\dots,m$.  Then, in order to show that $\alpha_0\in S$, it suffices to show that, for each $\alpha\in R$, there exists a positive integer $m$ such that $3^m\alpha\in S^9$.

Now, for any $t, \delta\in R$, we have $A(t,\delta):=\begin{pmatrix}
t & \delta \\
-1 & 0 
\end{pmatrix}\in M_2(R)$.  By (\ref{trf}), we compute directly that
\begin{equation}\label{trA9}
tr(A(t,\delta)^{9})=t^{9}-9t^{7}\delta+27t^{5}\delta^2-30t^{3}\delta^3+9t\delta^4 \in S^9, \;\forall t,\delta\in R.
\end{equation}
When $t=1$ and $\delta=0$, we have $1=tr(A(1,0)^{9})\in S$ and thus $\mathbb{Z}\subseteq S^9$ (because $S^9$ is a group). 
By setting $t=1, \delta=x$ (in (\ref{trA9})) equipped with $\mathbb{Z}\subseteq S^9$ , we have that 
$$P_0(x):=-9x+27x^2-30x^3+9x^4\in S^9, \;\forall x\in R.$$
Then $P_1(x):=P_0(-x)-P_0(x)=18x+60x^3\in S^9$ for all $x\in R$.  So, using $\mathbb{Z}\subseteq S^9$,
$P_2(x):=P_1(x+1)-P_1(x)\equiv 180x+180x^2\in S^9$ for all $x\in R$. This implies that   
$$P_3(x):=P_2(x)-P_2(-x)=360x\in S^9, \; \forall x\in R. $$ 
Then $P_4(x):=6P_1(x)-P_3(x^3)=108 x\in S^9$ for all $x\in R$.  Note that $$\gcd(P_3(x),P_4(x))=\gcd(360x,108x)=36x$$ and that doing Euclidean algorithm for $P_3(x)$ and $P_4(x)$ produces an element in the group $S^9$; so, 
$36 x \in S^9$ for all $x\in R$.  Thus, by using this fact together with $P_0(6x)\in S^9$  for all $x\in R$, we conclude that
$$R(x):=18x\in S,\;\; \forall x\in R.  $$
This also implies $[P_1(x)-R(x)]-3R(x^3)=6 x^3\in S^9$ for all $x\in R$.  Let $$K:=\mathbb{Z}\cup \{18x^i, 6x^{3i},x^{9i}\mid x\in R, i\in \mathbb{N}\}.$$  Then $K\subseteq S^9$. Hence, under this fact and (\ref{trA9}), we have  $$f(t,\delta):=9[t^7\delta+t^5\delta+t\delta^4] \in ^9; \; \forall t,\delta\in R.   $$
Based on $K\subseteq S$, we compute that:
\begin{eqnarray*}
	f(1,x) &\equiv_K & 9[x+x^2+x^4] \\
	f(x,1) &\equiv_K& 9[x+x^5+x^7] \\
	f(x,x^3+x) &\equiv_K & 9[x^5+x^7+x^8+x^{10}+x^{11}+x^{13}]
\end{eqnarray*}
belong to $S^9$, for all $x\in R$.  This yields
$$	f(x+x^2,1) \equiv_K 9[x+x^2+x^5+x^6+x^7+x^8+x^{11}+x^{12}+x^{13}+x^{14}]\in S^9, \; \forall x\in R. $$
Then, $$f(x,x^3+x)+f(x+x^2,1)+f(x^2,1)+f(1,x^3)\equiv _K 9[x+x^3]\in S^9, \; \forall x\in R.$$ 
Let $g(x):=9[x+x^3]$.  Then $g(x)\in S^9$ for all $x\in R$. So, based on $K\subseteq S^9$, we have that
$$ h(x):=9[x+x^2]\equiv_K g(x+1)-g(x) \in S^9,\; \forall x\in R.$$
Therefore, $3^2x=9x \equiv_K h(x^2)-f(1,x)\in S^9$ for all $x\in R$, which completes the proof for the equivalence of (2) and (3).  The assertion for an order in an algebraic number field  is an immediate consequence of Theorem \ref{odercri} and the equivalence of one of (1)-(2) and (3).

The proof of the statement (4) follows from the statement (2), Theorem \ref{prop32} and the fact that $\{a_0^9+3a_1^3\; (\operatorname{mod} 9R )\mid a_0, a_1\in R\}$ is an additive group.

To prove the statement (5), let $M\in M_n(R)$ with $tr(M)=a_0^9+3a_1^3+9a_2\in W(3,2,R)$, where $a_0, a_1, a_2\in R$.  Using the same arguments as in the proof of the statement (1), it is enough to show that $3\alpha^3$ and $9\alpha$ belong to the additive groups $S^9$.  Since we already have that $f(\alpha):=9\alpha\in S^9$ and $g(\alpha):=6\alpha^3\in S^9$ for all $\alpha \in R$.  Hence, $3\alpha^3=f(\alpha^3)-g(\alpha)$ and $9\alpha$ are immediately shown to be in $S^9$ for every $\alpha\in R$.

\subsection{Proof of Theorem \ref{maindegcomposite} for $k=10$}  We first show that $$S_{10}:=\{x_0^{10}-2x_1^5+5x_2^2 \mod 10 R\;|\; x_0,x_1,x_2\in R\}$$
is a group under the addition modulo $10R$ of the ring $R$.  It is clear that $0\in S_{10}$ and $-\alpha \mod 10R=9\alpha \mod 10 R$ for any $\alpha\in S_{10}$, then it remains to show that $S_{10}$ is closed.  We observe from the direct calculation that, for each $x,y \in R$,
\begin{eqnarray*}
	(x+y)^{10} \mod 10 R &\equiv &x^{10}+y^{10}+5(x^8y^2+x^2y^8)+2x^5y^5 \\
	2(x+y)^{5} \mod 10 R &\equiv &2(x^{5}+y^{5})\\
	5(x+y)^{2} \mod 10 R &\equiv &5(x^{2}+y^{2}).
\end{eqnarray*}
Let $\alpha:=x_0^{10}-2x_1^5+5x_2^2 \mod 10 R$ and $\beta:=y_0^{10}-2y_1^5+5y_2^2 \mod 10 R$ be in $S_{10}$.  By using the above relations, we compute that, in modulo $10 R$,
\begin{eqnarray*}
	\alpha+\beta  &\equiv &(x_0^{10}+y_0^{10})-2(x_1^5+y_1^5)+5(x_2^2+y_2^2) \\
	&\equiv &[(x_0+y_0)^{10}-2(x_0y_0)^5+5((x_0^4y_0)^2+(x_0y_0^4)^2)]-2(x_1+y_1)^5+5(x_2+y_2)^2\\
	&\equiv &(x_0+y_0)^{10}-2(x_0y_0+x_1+y_1)^5+5((x_0^4y_0)^2+(x_0y_0^4)^2+x_2+y_2)^2,
\end{eqnarray*}
which means that $\alpha+\beta\in S_{10}$ and thus $(S_{10},+)$ is a group as desired.

Now, suppose that $M\in M_2(R)$ is a sum of $10$-th powers of matrices in $M_2(R)$; namely $M=\sum_{i=1}^kM_i^{10}$ for some $k\in \mathbb{N}$ and $M_i\in M_2(R)$.  Then,  $tr(M) =\sum_{i=1}^ktr(M_i)$  and, by (\ref{trf}), for each $i\in\{1,\dots,k\}$,
\begin{eqnarray*}
	tr(M_i) \mod 10 R 	&\equiv &t_i^{10}-10t_i^8\delta_i+35t_i^6\delta_i^2-50t_i^4\delta_i^3+25t_i^2\delta_i^4-2\delta_i^5\\
	&\equiv&t_i^{10}+5t_i^6\delta_i^2+5t_i^2\delta_i^4-2\delta_i^5\\
	&\equiv&t_i^{10}-2\delta_i^5+5(t_i^3\delta_i+t_i\delta_i^2)^2,
\end{eqnarray*}
where $t_i$ and $\delta_i$ are trace and determinant of $M_i$ respectively.  So, $tr(M_i) \mod 10 R \in S_{10}$ for each $i\in\{1,\dots,k\}$.  Since $S_{10}$ is a group, $tr(M) \mod 10 R \in S_{10}$.

Conversely, suppose that $M\in M_2(R)$ with $tr(M)\mod 10R\in S_{10}$; namely $$tr(M):=x_0^{10}-2x_1^5+5x_2^2+10x_3$$ for some $x_0,x_1,x_2,x_3\in R$.  Let $ S^{10} $  be a subgroup of $(R,+)$ generated by the trace of 10-th powers of matrices in $M_2(R)$.  In order to show that $M$ is a sum of 10-th powers of matrices in $M_2(R)$, by the equivalent statements of (1) and (3) in Theorem \ref{basicthm3.1}, it suffices to show that $tr(M)\in S^{10}$.  By setting $$M_1:=\begin{pmatrix}
x_0 & 0 \\
-1 & 0 
\end{pmatrix} \hbox{ and } M_2:=\begin{pmatrix}
0 & x_1 \\
-1 & 0 
\end{pmatrix},$$ by (\ref{trf}), we have $tr(M_1^{10})=x_0^{10}$ and $tr(M_2^{10})=-2x_1^{5}$.  Since $(S^{10},+)$ is a group, it remains to show that $5x^2$ and $10x\in S^{10}$ for any $x\in R$.  We again apply  (\ref{trf}) to the matrices $$A:=\begin{pmatrix}
t & \delta \\
-1 & 0 
\end{pmatrix}$$
for which we have 
$$ p_0(t,\delta):=tr(A^{10})=t^{10}-10t^8\delta+35t^6\delta^2-50t^4\delta^3+25t^2\delta^4-2\delta^5\in S^{10} ,$$
for all $t,\delta\in R$.  By setting $t=1,\delta=0$, we have $1\in S^{10}$ and hence $\mathbb{Z}\subseteq S^{10}$.  Also, $x^{10}=p_0(1,0)\in S^{10}$ and $2x^{5}=p_0(0,-x)\in S^{10}$ for any $x\in R$.  Thus, $$K_1:=\mathbb{Z}\cup \{x^{10i},2x^{5i}\mid x\in R, i\in \mathbb{N}\}\subseteq S^{10}.$$
By using $K_1$, we reduce $p_0(-1,x)$ to be $$p_1(x):=25x^4+50x^3+35x^2+10x\in S^{10},\;\;\;\hbox{ for all $x\in R$.}$$
By using $2(x+1)^5\in S^{10}$ and $K_1$, we have that
$$p_2(x):=10x^4+20x^3+20x^2+10x\in S^{10},\;\;\;\hbox{ for all $x\in R$.}$$
So, $p_3(x):=5p_2(x)-2p_1(x)=30x^2+30x\in S^{10}$, for all $x\in R$, which also implies that $p_3(x+1)-p_3(x)=60x\in S^{10}$, for all $x\in R$.  Thus
$$ K_2:=K_1\cup \{60x^i\;|\; x\in R, i\in \mathbb{N}\}\subseteq S^{10}. $$
By using $K_2 $ and $p_3(x)$, we reduce $p_1(x)$ to $p_4(x):=5x^4+10x^3-5x^2-10x\in S^{10}$ and thus $$p_5(x):=p_3(x)-p_3(-x)\equiv_{K_2} 20x^3+40x  $$
and $$p_6(x):=p_3(x)+p_3(-x)\equiv_{K_2} 10x^4-10x^2  $$
belong to $S^{10}$ for all $x\in R$.

Now, by using $p_0(x,-1)$ and $K_1$, we have $g(x):=10x^8+35x^6+10x^4+25x^2$ is in $S^{10}$ for all $x\in R$.  Then, $$g_1(x):=g(x)-p_6(x^2)-2p_5(x^2)=5x^6-20x^4-5x^2\in S^{10}\;\;\; \hbox{ for all $x\in R$}.$$Then,
\begin{eqnarray*}
	g_1(x+1)-g_1(x)&=&30x^5+55x^4+20x^3-50x^2-60x-20 \\
	&\equiv_{K_2}&15x^4+20x^3+15x^2:=g_2(x)\in S^{10}, \;\; \forall x\in R.
\end{eqnarray*}
This yields that $40x^3=g_2(x)-g_2(-x)\in S^{10}$ and thus $20x=2p_5(x)-40x^3-60x\in S^{10}$; namely $20x^i\in S^{10}$ for all $x\in R$ and $i\in \mathbb{N}$.  Moreover,  $$g_3(x):=20x^4+20x^2-g_2(x)-20x^3=5x^4+5x^2\in S^{10}, \;\; \forall x\in R.$$ 
Also, $10x^3=p_1(x)-5g_3(x)-p_3(x)+20x^2+20x-40x^3\in S^{10}$; namely $10x^{3i}\in S^{10}$ for all $x\in R$ and $i\in \mathbb{N}$.  Furthermore, by using $(x+1)^{10}\in S^{10}$ and $$K_3:=K_2\cup \{20x^i, 10x^{3i}, 5x^{4i}+5x^{2i}\;|\; x\in R, i\in \mathbb{N}\}\subseteq S^{10},$$  we see that 
\begin{eqnarray*}
	(x+1)^{10}&\equiv_{K_3}&10x^9+5x^8+10x^6+10x^4+5x^2+10x \\
	&\equiv_{K_3}&5x^8+10x^4+5x^2+10x \\
	&\equiv_{K_3}&10x;
\end{eqnarray*}
namely, $10x\in S^{10}$ for all $x\in R$.  By setting,
$$K_4:=K_3\cup \{10x^i\;|\; x\in R, i\in \mathbb{N}\}\subseteq S^{10},$$
we can deduce that  $p_0(t,\delta)\equiv_{K_4}5t^6\delta^2+5t^2\delta^4$. Replacing $\delta:=x^2$ and $t:=x+x^2$, for $x\in R$, we obtain that 
\begin{eqnarray*}
	5(x+x^2)^{6}x^4+5(x+x^2)^2x^8&\equiv_{K_4}&5x^{16}+5x^{14}:=f(x).
\end{eqnarray*}
So, 
\begin{eqnarray*}
	f(x+1)&\equiv_{K_4}&5[x^{16}+x^{14}+x^{12}+x^{10}+x^8+x^6+x^4+x^2] \\
	&\equiv_{K_4}&5x^2\in S^{10}\;\; \forall x\in R.
\end{eqnarray*}
Hence, $tr(M):=x_0^{10}-2x_1^5+5x_2^2+10x_3\in S^{10}$ for any $x_0,x_1,x_2,x_3\in R$ and the statement (1) is proved.

To prove the statement (3), we use the statement (2).   Precisely, if $tr(M) \mod 10R\in S_{10}$, then (by statement (2)) there exist $M_i\in M_2(R)$ such that $tr(M)=\sum_{i=1}^r tr(M_i^{10})$.  For each $i$, let $N_i:=M_i\oplus O_{n-2}\in M_n(R)$.  This yields that $tr(N_i^{10})=tr(M_i^{10})$ and that $tr(M)=\sum_{i=1}^r tr(N_i^{10})$.  By the equivalence of the statements (1) and (3) in Theorem \ref{basicthm3.1}, the proof is completed.

\subsection{Proof of Theorem \ref{maindegcomposite} for $k=14$}  We first show that $$S_{14}:=\{x_0^{14}-2x_1^7+7x_2^2 \mod 14 R\;|\; x_0,x_1,x_2\in R\}$$
is a group under the addition modulo $14R$ of the ring $R$.  It is clear that $0\in S_{14}$ and $-\alpha \mod 14R=13\alpha \mod 14 R$ for any $\alpha\in S_{14}$, then it remains to show that $S_{14}$ is closed.  We observe from the direct calculation in modulo $14R$ that, for each $x,y \in R$,
\begin{eqnarray*}
	(x+y)^{14} &\equiv &x^{14}+y^{14}+7(x^{12}y^2+x^2y^{12})+7(x^{10}y^4+x^4y^{10})+7(x^{8}y^6+x^6y^{8})+2x^7y^7 \\
	2(x+y)^{7} &\equiv &-2(x^{7}+y^{7})\\
	7(x+y)^{2}&\equiv &7(x^{2}+y^{2}).
\end{eqnarray*}
Let $\alpha:=x_0^{14}-2x_1^7+7x_2^2 \mod 14 R$ and $\beta:=y_0^{14}-2y_1^7+7y_2^2 \mod 14 R$ be in $S_{14}$.  By using the above relations, we compute that, in modulo $14 R$,
\begin{eqnarray*}
	\alpha+\beta  &\equiv &(x_0^{14}+y_0^{14})-2(x_1^7+y_1^7)+7(x_2^2+y_2^2) \\
	&\equiv &[(x_0+y_0)^{14}-2(x_0y_0+x_1y_1)^7 \\
	&+&7(x_0^6y_0+x_0y_0^6+x_0^5y_0^2+x_0^2y_0^5+x_0^4y_0^3+x_0^3y_0^4+x_2+y_2)^2
\end{eqnarray*}
which means that $\alpha+\beta\in S_{14}$ and thus $(S_{14},+)$ is a group as desired.

Now, suppose that $M\in M_2(R)$ is a sum of $14$-th powers of matrices in $M_2(R)$; namely $M=\sum_{i=1}^kM_i^{14}$ for some $k\in \mathbb{N}$ and $M_i\in M_2(R)$.  Then,  $tr(M) =\sum_{i=1}^ktr(M_i)$  and, by (\ref{trf}), for each $i\in\{1,\dots,k\}$,
\begin{eqnarray*}
	tr(M_i) \mod 14 R 	&\equiv &t_i^{14}-14t_i^{12}\delta_i+77t_i^{10}\delta_i^2-210t_i^8\delta_i^3+294t_i^6\delta_i^4-196t_i^4\delta_i^5+49t_i^2\delta_i^6-2\delta_i^7\\
	&\equiv&t_i^{14}+7t_i^{10}\delta_i^2+7t_i^2\delta_i^6-2\delta_i^7\\
	&\equiv&t_i^{14}-2\delta_i^7+7(t_i^5\delta_i+t_i\delta_i^3)^2,
\end{eqnarray*}
where $t_i$ and $\delta_i$ are trace and determinant of $M_i$ respectively.  So, $tr(M_i) \mod 14 R \in S_{14}$ for each $i\in\{1,\dots,k\}$.  Since $S_{14}$ is a group, $tr(M) \mod 14 R \in S_{14}$.

Conversely, suppose that $M\in M_2(R)$ with $tr(M)\mod 14R\in S_{14}$; namely $$tr(M):=x_0^{14}-2x_1^7+7x_2^2+14x_3$$ for some $x_0,x_1,x_2,x_3\in R$.  Let $ S^{14} $  be a subgroup of $(R,+)$ generated by the trace of 14-th powers of matrices in $M_2(R)$.  In order to show that $M$ is a sum of 14-th powers of matrices in $M_2(R)$, by the equivalent statements of (1) and (3) in Theorem \ref{basicthm3.1}, it suffices to show that $tr(M)\in S^{14}$.  By setting $$M_1:=\begin{pmatrix}
x_0 & 0 \\
-1 & 0 
\end{pmatrix} \hbox{ and } M_2:=\begin{pmatrix}
0 & x_1 \\
-1 & 0 
\end{pmatrix},$$ by (\ref{trf}), we have $tr(M_1^{14})=x_0^{14}$ and $tr(M_2^{14})=-2x_1^{7}$.  Since $(S^{14},+)$ is a group, it remains to show that $7x^2$ and $14x\in S^{14}$ for any $x\in R$.  We again apply  (\ref{trf}) to the matrices $$A:=\begin{pmatrix}
t & \delta \\
-1 & 0 
\end{pmatrix}$$
for which we have 
$$ p_0(t,\delta):=tr(A^{14})=t^{14}-14t^{12}\delta+77t^{10}\delta^2-210t^8\delta^3+294t^6\delta^4-196t^4\delta^5+49t^2\delta^6-2\delta^7\in S^{14} ,$$
for all $t,\delta\in R$.  By setting $t=1,\delta=0$, we have $1\in S^{14}$ and hence $\mathbb{Z}\subseteq S^{14}$.  Also, $x^{14}=p_0(1,0)\in S^{14}$ and $2x^{7}=p_0(0,-x)\in S^{14}$ for any $x\in R$.  Thus, $$K_1:=\mathbb{Z}\cup \{x^{14i},2x^{7i}\mid x\in R, i\in \mathbb{N}\}\subseteq S^{14}.$$
By using $K_1$, we reduce $p_0(-1,x)$ to be $$p_1(x):=49x^6+196x^5+294x^4+210x^3+77x^2+14x\in S^{14},\;\;\;\hbox{ for all $x\in R$.}$$
By using $2(x+1)^7\in S^{14}$ and $K_1$, we have that
$$p_2(x):=14x^6+42x^5+70x^4+70x^3+42x^2+14x\in S^{14},\;\;\;\hbox{ for all $x\in R$.}$$
We now calculate that
\begin{eqnarray*}
	p_3(x):=2p_1(x)-7p_2(x)&\equiv_{K_1}&98x^5+98x^4-70x^3-140x^2-70x\\
	p_5(x):=p_1(x)-p_1(-x)&\equiv_{K_1}&392x^5+420x^3+28x\\
	p_6(x):=p_2(x)-p_2(-x)&\equiv_{K_1}&84x^5+70x^3+28x\\
	p_7(x):=4p_3(x)-p_5(x)&\equiv_{K_1}&392x^4-700x^3-140x^2-308x\\
	p_8(x):=p_2(x+1)-p_2(x)&\equiv_{K_1}&84x^5+420x^4+980x^3+1260x^2+868x\\
	p_9(x):=p_8(x)-p_6(x)&\equiv_{K_1}&420x^4+910x^3+1260x^2+840x
\end{eqnarray*}
and further calculate that
\begin{eqnarray*}
	p_{10}(x):=p_9(x)-p_9(-x)+p_7(x)-p_7(-x)&\equiv_{K_1}&420x^3+1064x\\
	p_{11}(x):=p_3(x+1)-p_3(x)-p_9(x)&\equiv_{K_1}&70x^4+462x^3+98x^2-448x\\
	p_{12}(x):=p_{11}(x)+p_{11}(-x)&\equiv_{K_1}&140x^4+196x^2\\
	p_{13}(x):=p_{11}(x+1)-p_{11}(-x)&\equiv_{K_1}&280x^3+1806x^2+1862x\\
	p_{14}(x):=p_{13}(x+1)-p_{13}(-x)&\equiv_{K_1}&840x^2+4452x
\end{eqnarray*}
and again further calculate that
\begin{eqnarray*}
	p_{15}(x):=4p_{14}(x)-[p_{14}(x^2)-6p_{12}(x)]&\equiv_{K_1}&84x^2+17808x\\
	p_{16}(x):=p_{15}(x+1)-p_{15}(x)&\equiv_{K_1}&168x\\
	p_{17}(x):=p_{15}(x)-106p_{16}(x)&\equiv_{K_1}&84x^2\\
	p_{18}(x):=p_{14}(x)-10p_{17}(x)-26p_{16}(x)&\equiv_{K_1}&84x\\
	p_{19}(x):=p_{10}(x)-6[p_{6}(x)-p_{18}(x^5)]-10p_{18}(x)&\equiv_{K_1}&56x
\end{eqnarray*}
which we can conclude that $28x=p_{18}(x)-p_{19}(x)$ belongs to $S^{14}$ for all $x\in R$.  Let $K_2:=K_1\cup\{28x^i\;|\;x\in R, i\in \mathbb{N}\}$.  Then $K_2\subseteq S^{14}$ and $-p_1(x)$ is reduced to (by $K_2$)
$$g_1(x):=7x^6+14x^4-7x^2+14x\in S^{14},\;\; \forall x\in R,  $$
and $p_{13}(x)$ is reduced to
$$g_2(x):=14x^2+14x \in S^{14},\;\; \forall x\in R$$
and $p_6(x)$ is reduced to 
$$g_3(x):=14x^3 \in S^{14},\;\; \forall x\in R.$$
Then, $7x^6-7x^2=g_1(x)-g_2(x^2)-g_2(x)\in S^{14}$ and $7x^6+7x^2=g_3(x^2)-(7x^6-7x^2)\in S^{14}$ for all $x\in R$.  Thus, $14x^2\in S^{14}$  for all $x\in R$.  By $g_2(x)$, we can reach to the conclusion that $14x\in S^{14}$ for all $x\in R$. 

Now, it remains to show that $7x^2\in S^{14}$ for all $x\in R$, for which we can  compute in modulo $$K_3:=K_2\cup\{14x^{i}\;|\;x\in R, i\in \mathbb{N}\}\subseteq S^{14}.$$
Under $K_3$, $p_0(t,\delta)$ is reduced to
$$f(t,\delta) :=7[t^{10}\delta^2+7t^2\delta^6]\in S^{14},\;\; \forall t,\delta\in R. $$
Then, 
\begin{eqnarray*}
	f(x+x^2,x)&\equiv_{K_3}& 7[x^{22}+x^{20}+x^{14}+x^{12}]\\
	&\equiv_{K_3}& 7[x^{22}+x^{20}+x^{12}]\in S^{14} \;\; \forall x\in R
\end{eqnarray*}
and
$$f(x^2,x) \equiv_{K_3}7[x^{22}+x^{20}]\in S^{14} \;\; \forall x\in R.$$
Hence $$f(x+x^2,x)+f(x^2,x)\equiv_{K_3} 7x^{12}\equiv_{K_3} 7x^6\in S^{14}\;\; \forall x\in R.$$
Since $7x^6-7x^2\in S^{14}$, we can now conclude that $7x^2\in S^{14}$ for all $x\in R$, which completes the proof for the statement (2). The statement (3) follows by using the same arguments as in that proof of the case $k=10$.

\subsection{Proof of Theorem \ref{maindegcomposite} for $k=15$}  We first show that $$S_{15}:=\{x_0^{15}-3x_1^5+5x_2^3 \mod 15 R\;|\; x_0,x_1,x_2\in R\}$$
is a group under the addition modulo $15R$ of the ring $R$.  It is clear that $0\in S_{15}$ and $-\alpha \mod 15R=14\alpha \mod 15 R$ for any $\alpha\in S_{15}$, then it remains to show that $S_{15}$ is closed.  We observe from the direct calculation in modulo $15R$ that, for each $x,y \in R$,
\begin{eqnarray*}
	(x+y)^{15} &\equiv &x^{15}+y^{15}+5((x^{4}y)^3+(xy^4)^3)+3((x^{2}y)^5+(xy^2)^5)+10((x^{3}y^2)^3+(x^2y^3)^3)\\
	3(x+y)^{5} &\equiv &3(x^{5}+y^{5})\\
	5(x+y)^{3}&\equiv &5(x^{3}+y^{3}).
\end{eqnarray*}
Let $\alpha:=x_0^{15}-3x_1^5+5x_2^3 \mod 15 R$ and $\beta:=y_0^{15}-3y_1^5+5y_2^3 \mod 15 R$ be in $S_{15}$.  By using the above relations, we compute that, in modulo $15 R$,
\begin{eqnarray*}
	\alpha+\beta  &\equiv &(x_0^{15}+y_0^{15})-3(x_1^5+y_1^5)+5(x_2^3+y_2^3) \\
	&\equiv &(x_0+y_0)^{15}-3(x_0^2y_0+x_0y^2+x_1+y_1)^5 \\
	&+&5(x_0^3y_0^2+x_0^2y_0^3-x_0^4y_0-x_0y_0^4+x_2+y_2)^3
\end{eqnarray*}
which means that $\alpha+\beta\in S_{15}$ and thus $(S_{15},+)$ is a group as desired.

Now, suppose that $M\in M_2(R)$ is a sum of $15$-th powers of matrices in $M_2(R)$; namely $M=\sum_{i=1}^kM_i^{15}$ for some $k\in \mathbb{N}$ and $M_i\in M_2(R)$.  Then,  $tr(M) =\sum_{i=1}^ktr(M_i)$  and, by (\ref{trf}), for each $i\in\{1,\dots,k\}$, (in modulo $15R$),
\begin{eqnarray*}
	tr(M_i) &\equiv &t_i^{15}-15t_i^{13}\delta_i+90t_i^{11}\delta_i^2-275t_i^9\delta_i^3+450t_i^7\delta_i^4-378t_i^5\delta_i^5+140t_i^3\delta_i^6-15t\delta_i^7\\
	&\equiv&t_i^{15}-5t_i^{9}\delta_i^3-3t_i^5\delta_i^5+5t_i^3\delta_i^6\\
	&\equiv&t_i^{15}-3(t_i\delta_i)^5+5(t_i\delta_i^2-t_i^{3}\delta_i)^3
\end{eqnarray*}
where $t_i$ and $\delta_i$ are trace and determinant of $M_i$ respectively.  So, $tr(M_i) \mod 15 R \in S_{15}$ for each $i\in\{1,\dots,k\}$.  Since $S_{15}$ is a group, $tr(M) \mod 15 R \in S_{15}$.

Conversely, suppose that $M\in M_2(R)$ with $tr(M)\mod 15R\in S_{15}$; namely $$tr(M):=x_0^{15}-3x_1^5+5x_2^3+15x_3$$ for some $x_0,x_1,x_2,x_3\in R$.  Let $ S^{15} $  be a subgroup of $(R,+)$ generated by the trace of 15-th powers of matrices in $M_2(R)$.  In order to show that $M$ is a sum of 15-th powers of matrices in $M_2(R)$, by the equivalent statements of (1) and (3) in Theorem \ref{basicthm3.1}, it suffices to show that $tr(M)\in S^{15}$.  By setting $$M_1:=\begin{pmatrix}
x_0 & 0 \\
-1 & 0 
\end{pmatrix}$$
for which we have 
$$ p_0(t,\delta):=tr(A^{15})=t^{15}-15t^{13}\delta+90t^{11}\delta^2-275t^9\delta^3+450t^7\delta^4-378t^5\delta^5+140t^3\delta^6-15t\delta^7\in S^{15} ,$$
for all $t,\delta\in R$.  By setting $t=1,\delta=0$, we have $1\in S^{15}$ and hence $\mathbb{Z}\subseteq S^{15}$.  Also, $x^{15}=p_0(1,0)\in S^{15}$ for any $x\in R$.  Thus, $$K_1:=\mathbb{Z}\cup \{x^{15i}\mid x\in R, i\in \mathbb{N}\}\subseteq S^{15}.$$
By using $K_1$, we reduce $p_0(-1,x)$ to be $$p_0(x):=15x^7+140x^6+378x^5+450x^4+275x^3+90x^2+15x\in S^{15},\;\;\;\hbox{ for all $x\in R$.}$$
We now calculate that
\begin{eqnarray*}
	p_1(x):=p_0(x+1)-p_0(x)&\equiv_{K_1}&105x^6+1155x^5+4515x^4+8905x^3+9720x^2+5640x\\
	p_2(x):=p_1(x)+p_1(-x)&\equiv_{K_1}&210x^6+9030x^4+19440x^2\\
	p_3(x):=p_2(x+1)-p_2(x)&\equiv_{K_1}&30[42x^5+105x^4+1344x^3+1911x^2+2542x]\\
	p_4(x):=p_3(x)+p_3(-x)&\equiv_{K_1}&60[105x^4+1911x^2]\\
	p_5(x):=p_4(x+1)-p_4(x)&\equiv_{K_1}&60\cdot42[10x^3+15x^2+101x]\\
	p_6(x):=p_5(x)+p_5(-x)&\equiv_{K_1}&120\cdot42[15x^2]
\end{eqnarray*}
and further calculate that
\begin{eqnarray*}
	p_{7}(x):=12p_4(x)-p_6(x^2)-18p_6(x)&\equiv_{K_1}&720\cdot21x^2\\
	p_{8}(x):=p_2(x)-p_7(x^3)-43p_7(x^2)-92p_7(x)&\equiv_{K_1}&720\cdot 12x^2\\
	p_{9}(x):=2p_{8}(x)-p_{7}(x)&\equiv_{K_1}&36\cdot 60x^2\\
	p_{10}(x):=3p_{9}(x^2)-p_{4}(x)+53p_9(x)&\equiv_{K_1}&180x^4-180x^2\\
	p_{11}(x):=2[p_{0}(x)+p_{0}(-x)+p_{10}(x)]-p_9(x^2)&\equiv_{K_1}&560x^6
\end{eqnarray*}
and again further calculate that
\begin{eqnarray*}
	p_{12}(x):=4p_{11}(x)-p_{9}(x^3)&\equiv_{K_1}&80x^6\\
	p_{13}(x):=p_{2}(x)-9p_{9}(x)-4p_9(x^2)-2p_{12}(x)&\equiv_{K_1}&390x^4+50x^6\\
	p_{14}(x):=8p_{13}(x)-5p_{12}(x)&\equiv_{K_1}&240\cdot 13x^4\\
	p_{15}(x):=3p_{9}(x^2)-2p_{14}(x)&\equiv_{K_1}&240x^4\\
	p_{16}(x):=p_{13}(x)-p_{15}(x)&\equiv_{K_1}&150x^4+50x^6
\end{eqnarray*}
which we still further calculate that
\begin{eqnarray*}
	p_{17}(x):=[p_{0}(x)+p_{0}(-x)]-6p_{16}(x)+p_{12}(x)&\equiv_{K_1}&60x^6+180x^2\\
	p_{18}(x):=4p_{17}(x)-3p_{12}(x)&\equiv_{K_1}&540x^2\\
	p_{19}(x):=p_{12}(x)-[p_{18}(x^2)-2p_{15}(x)]&\equiv_{K_1}&20x^4\\
	p_{20}(x):=9p_{19}(x)-p_{10}(x)&\equiv_{K_1}&180x^2\\
	p_{21}(x):=7p_{12}(x)-p_{18}(x^2)&\equiv_{K_1}&20x^6
\end{eqnarray*}
and that
\begin{eqnarray*}
	p_{22}(x):=p_{21}(x+1)-p_{21}(x)&\equiv_{K_1}&120x^5+400x^3+300x^2+120x\\
\end{eqnarray*}
which yields
\begin{eqnarray*}
	p_{23}(x):=p_{22}(x)+p_{22}(-x)-p_{18}(x)&\equiv_{K_1}&60x^2\\
	p_{24}(x):=p_{23}(x+1)-p_{23}(x)]&\equiv_{K_1}&120x\\
	p_{25}(x):=p_{22}(x)-p_{24}(x^5)-3p_{24}(x^3)-5p_{23}(x)-p_{24}(x)&\equiv_{K_1}&40x^3
\end{eqnarray*}
belong to $S^{15}$ for all $x\in R$.  Moreover, by using $$K_2:=K_1\cup\{120x^i,60x^{2i},40x^{3i},20x^{4i},20x^{6i}\;|\;x\in R, i\in \mathbb{N}\}\subseteq S^{15},$$
we have that
\begin{eqnarray*}
	p_{0}(x)-p_{0}(-x)&\equiv_{K_2}&30x^7+756x^5+550x^3+30x\\
	&\equiv_{K_2}&30x^7+36x^5+10x^3+30x =:p_{26}(x)\in S^{15}\;\; \forall x\in R.
\end{eqnarray*}
Then, 
\begin{eqnarray*}
	p_{27}(x):=p_{26}(2x)&\equiv_{K_2}&72x^5+60x\\
	p_{28}(x):=p_{0}(2x)&\equiv_{K_2}&96x^5+30x\\
	p_{29}(x):=p_{28}(x)-p_{27}(x)&\equiv_{K_2}&24x^5-30x\\
	p_{30}(x):=p_{29}(5x)&\equiv_{K_2}&30x\\
	p_{31}(x):=p_{25}(x)-p_{30}(x^3)&\equiv_{K_2}&10x^3
\end{eqnarray*} 
and  
$$p_{32}(x):=p_{26}(x)-p_{30}(x^7)-p_{30}(x^5)-p_{31}(x)-p_{30}(x)\equiv_{K_2}6x^5, $$
which implies that $$K_3:=K_2\cup\{30x^i,10x^{3i},6x^{5i}\;|\; x\in R, i\in \mathbb{N}\}  \subseteq S^{15}.$$
By using $K_3\subseteq S^{15}$, $p_0(t,\delta)$ is reduced to
$$ f(t,\delta):=15[t^{13}\delta+t^9\delta^3+t\delta^7]\in S^{15}, \;\; \forall x\in R. $$
So, 
\begin{eqnarray*}
	f(1,x)&\equiv_{K_3}&15[x^7+x^3+x]\\
	f(x,1)&\equiv_{K_3}&15[x^{13}+x^9+x]\\
	f(x,x)&\equiv_{K_3}&15[x^{14}+x^{12}+x^8]\\
	g_0(x):=(x+1)^{15}&\equiv_{K_3}&15[x^{12}+x^{11}+x^{10}+x^8+x^7+x^5+x^4+x^3]
\end{eqnarray*} 		
and then	
\begin{eqnarray*}
	g_1(x):=f(1,x+1)-f(1,x
	)&\equiv_{K_3}&15[x^6+x^5+x^4+x^3]\\
	g_2(x):=g_0(x)-g_1(x^2)-g_1(x)	&\equiv_{K_3}&15[x^{11}+x^7]\\
	g_3(x)=	f(x,x)+f(1,x^2)&\equiv_{K_3}&15[x^{12}+x^{8}+x^6+x^2]\\
	g_4(x):=g_3(x+1)-g_3(x)&\equiv_{K_3}&15[x^8
	+x^2]\\
	g_5(x):=g_3(x)-g_4(x)&\equiv_{K_3}&15[x^{12}
	+x^6]\\
	g_7(x):=g_1(x^2)-g_5(x)&\equiv_{K_3}&15[x^{10}
	+x^8]
\end{eqnarray*} 		
belong to $S^{15}$ for all $x\in R$.  Moreover,
\begin{eqnarray*}
	g_9(x):=g_2(x+1)-g_2(x)-g_1(x)&\equiv_{K_3}&15[x^{10}+x^9+x^8+x^3]\\
	g_{10}(x):=g_9(x)-g_7(x)&\equiv_{K_3}&15[x^{9}+x^3]\\
	g_{11}(x)=g_1(x^3)&\equiv_{K_3}&15[x^{18}+x^{12}+x^9]\\
	g_{12}(x):=g_{11}(x)-g_{10}(x^2)+g_5(x)&\equiv_{K_3}&15[x^9],
\end{eqnarray*} 		
which implies that $15x^3=g_{10}(x)-g_{10}(x)\in S^{15}$ for all $x\in R$.  Since $10x^3\in S^{15}$ for all $x\in R$, we now conclude that $$5x^3\in S^{15},\;\; \forall x\in R.$$
This also implies that $15[x^2+x]\equiv_{K_3} 5(x+1)^3-5x^3\in S^{15}$ for all $x\in R$.  Then
$$ K_4:=K_3\cup\{ 5x^{3i}, 15[x^i+x^{2i}]\;|\; x\in R, \in \mathbb{N}\}\subseteq S^{15}.$$
By using $K_4$, $f(t,\delta)$ is reduced to $h(t,\delta):=15[t^{13}\delta+t\delta^7]\in S^{15}$ and $f(1,x)+f(x,1)$ is reduced to $h_1(x):=15[x^{13}+x^7]\in S^{15}$ for all $t,\delta,x\in R$.  So,
$$ h(x,x^3+1) \equiv_{K_4} 15[x^{22}+x^{16}+x^{15}+x^{13}+x^{8}+x] \equiv_{K_4} 15[x^{13}+x^{11}+x]\in S^{15}.$$
Thus, $$15x\equiv_{K_4} h(x,x^3+1)+h_1(x)-g_2(x)\in S^{15},\; \; \forall x\in R.$$
Finally, since $6x^5\in S^{15}$ and $15x^5\in S^{15}$, $$3x^5=15x^5-2(6x^5) \in S^{15},\; \; \forall x\in R, $$
which completes the proof for the statement (2). The statement (3) follows by using the same arguments as in that proof of the case $k=10$.

\subsection{Proof of Theorem \ref{maindegcomposite} for $k=12$}  We first show that $$S_{12}:=\{x_0^{12}+2x_1^6-3x_2^4-4x_3^3+6x_4^2 \mod 12 R\;|\; x_0,x_1,x_2,x_3,x_4\in R\}$$
is a group under the addition modulo $12R$ of the ring $R$.  It is clear that $0\in S_{12}$ and $-\alpha \mod 12R=11\alpha \mod 12 R$ for any $\alpha\in S_{12}$, then it remains to show that $S_{12}$ is closed.  We observe from the direct calculation in modulo $12R$ that, for each $x,y \in R$,
\begin{eqnarray*}
	(x+y)^{12} &\equiv &(x^{12}+y^{12})+6(x^{10}y^2+x^2y^{10})+4(x^{9}y^3+x^3y^9)+3(x^{8}y^4+x^4y^8)\\
	x+y)^{6} &\equiv &(x^{6}+y^{6})+6(x^5y+xy^5) -4(xy)^3+3(x^4y^2+x^2y^4)\\
	(x+y)^{4} &\equiv &(x^{4}+y^{4})+6(xy)^2+4(x^{3}y+xy^{3})\\
	(x+y)^{3}&\equiv &(x^{3}+y^{3})+3(x^{2}y+xy^{2})\\
	(x+y)^{2}&\equiv &(x^{2}+y^{2})+2(xy)
\end{eqnarray*}
Let $\alpha:=x_0^{12}+2x_1^6-3x_2^4-4x_3^3+6x_4^2  \mod 12 R$ and $\beta:=y_0^{12}+2y_1^6-3y_2^4-4y_3^3+6y_4^2 \mod 12 R$ be in $S_{12}$.  By using the above relations, we compute that, in modulo $12 R$,
\begin{eqnarray*}
	\alpha+\beta  &\equiv &(x_0^{12}+y_0^{12})+2(x_1^6+y_1^6)-3(x_2^4+y_2^4)-4(x_3^3+y_3^3)+6(x_4^2+y_4^2)\\
	&\equiv &(x_0^{12}+y_0^{12})+2(x_1^6+y_1^6)-3(x_2+y_2+x_0^2y_0+x_0y_0^2)^4-4(x_0^3y_0+x_0y_0^3+x_1y_1+x_3+y_3) ^3\\
	&+&6[x_0^5y_0+x_0y_0^5+x_1^2y_1+x_1y_1^2+(x_2+y_2)(x_0^2y_0+x_0y_0^2)+x_0^3y_0^3+x_2y_2+x_4+y_4]^2
\end{eqnarray*}
which means that $\alpha+\beta\in S_{12}$ and thus $(S_{12},+)$ is a group as desired.

Now, suppose that $M\in M_2(R)$ is a sum of $12$-th powers of matrices in $M_2(R)$; namely $M=\sum_{i=1}^kM_i^{12}$ for some $k\in \mathbb{N}$ and $M_i\in M_2(R)$.  Then,  $tr(M) =\sum_{i=1}^ktr(M_i)$  and, by (\ref{trf}), for each $i\in\{1,\dots,k\}$, (in modulo $12R$),
\begin{eqnarray*}
	tr(M_i) &\equiv &t_i^{12}-12t_i^{10}\delta_i+54t_i^{8}\delta_i^2-112t_i^6\delta_i^3+105t_i^4\delta_i^4-36t_i^2\delta_i^5+2\delta_i^6\\
	&\equiv&t_i^{12}+6t_i^{8}\delta_i^2-4t_i^6\delta_i^3-3t_i^4\delta_i^4+2\delta_i^6\\
	&\equiv&t_i^{12}+2\delta_i^6-4(t_i^2\delta_i)^3-3(t_i\delta_i)^4+6(t_i^{4}\delta_i)^2
\end{eqnarray*}
where $t_i$ and $\delta_i$ are trace and determinant of $M_i$ respectively.  So, $tr(M_i) \mod 12 R \in S_{15}$ for each $i\in\{1,\dots,k\}$.  Since $S_{12}$ is a group, $tr(M) \mod 12 R \in S_{12}$, which completes the proof for (4).

To prove the statement (5), it suffices to  calculate on $M\in M_2(R)$ with $tr(M)\in S^*_{12}$; namely $$tr(M):=x_0^{12}+2x_1^6+3x_2^4+8x_3^3+12x_4^2+24x_5$$ or
$$tr(M):=x_0^{12}+2x_1^6+3x_2^4+8x_3^3+4x_4^{2m+1}+6x_4^2+12x_4,$$
for some $x_0,x_1,x_2,x_3,x_4\in R$ and $m\in \mathbb{N}$.   Let $ S^{12} $  be a subgroup of $(R,+)$ generated by the trace of 12-th powers of matrices in $M_2(R)$.  In order to show that $M$ is a sum of 12-th powers of matrices in $M_2(R)$, by the equivalent statements of (1) and (3) in Theorem \ref{basicthm3.1}, it suffices to show that $tr(M)\in S^{12}$.  By setting $$M_1:=\begin{pmatrix}
x_0 & 0 \\
-1 & 0 
\end{pmatrix} \hbox{ and } M_2:=\begin{pmatrix}
0 & x_1 \\
-1 & 0 
\end{pmatrix},$$ by (\ref{trf}), we have $tr(M_1^{12})=x_0^{12}$ and $tr(M_2^{12})=2x_1^{6}$ .  Since $(S^{12},+)$ is a group, it remains to show that $3x^4, 8x^3,12x^2, 24x$ and $4x^{2m+1}+6x^2+12x$ belong to $ S^{12}$  for any $x\in R$ and $m\in \mathbb{N}$.  We again apply  (\ref{trf}) to the matrices $$A:=\begin{pmatrix}
t & \delta \\
-1 & 0 
\end{pmatrix}$$
for which we have 
$$ p_0(t,\delta):=tr(A^{12})=t^{12}-12t^{10}\delta+54t^{8}\delta^2-112t^6\delta^3+105t^4\delta^4-36t^2\delta^5+2t\delta^6\in S^{12} ,$$
for all $t,\delta\in R$.  By setting $t=1,\delta=0$, we have $1\in S^{12}$ and hence $\mathbb{Z}\subseteq S^{12}$.  Also, $x^{12}=p_0(1,0)$ and $2x^6=p_0(0,x)$ are in $S^{12}$ for any $x\in R$.  Thus, $$K_1:=\mathbb{Z}\cup \{x^{12i}, 2x^{6i}\mid x\in R, i\in \mathbb{N}\}\subseteq S^{12}.$$
By using $K_1$, we reduce $p_0(1,-x)$ to be $$p_0(x):=36x^5+105x^4+112x^3+54x^2+12x\in S^{12},\;\;\;\hbox{ for all $x\in R$.}$$
We now calculate that
\begin{eqnarray*}
	p_1(x):=p_0(x)-p_0(x)&\equiv_{K_1}&210x^4+108x^2\\
	p_2(x):=2(x+1)^6-2x^6&\equiv_{K_1}&12x^5+30x^4+40x^3+30x^2+12x\\
	p_3(x):=p_0(x)-3p_2(x)&\equiv_{K_1}&15x^4-8x^3-36x^2-24x\\
	p_4(x):=p_1(x)-14p_3(x)&\equiv_{K_1}&112x^3+612x^2+24x\\
	p_5(x):=p_3(x)+p_3(-x)&\equiv_{K_1}&30x^4-72x^2\\
	p_6(x):=p_1(x)-7p_5(x)&\equiv_{K_1}&612x^2
\end{eqnarray*}
and further calculate that
\begin{eqnarray*}
	p_{7}(x):=p_3(-x)-p_3(x)&\equiv_{K_1}&16x^3+48x\\
	p_{8}(x):=-[p_4(x)-7p_7(x)-p_6(x)]&\equiv_{K_1}&312x\\
	p_{9}(x):=2p_{8}(x^2)-p_{6}(x)&\equiv_{K_1}&12x^2\\
	p_{10}(x):=p_{9}(x+1)-p_{9}(x)&\equiv_{K_1}&24x\\
	p_{11}(x):=p_7(x)-2p_{10}(x)&\equiv_{K_1}&16x^3\\
	p_{12}(x):=p_1(x)-9p_{9}(x)-17p_9(x^2)&\equiv_{K_1}&6x^4
\end{eqnarray*}
are in $S^{12}$ for all $x\in R$.  Since $p_0(x,-1)=12x^{10}+54x^8+112x^6+105x^4+36x^2$, 
$$p_0(x,-1)-p_9(x^5) -9p_{12}(x^2)-8p_{11} (x^2)-3p_9(x)-17p_{12}(x)=3x^4\in S^{12}\;\; \forall x\in R;$$
which also implies that $$8x^3=5(3x^4)-p_3(x)+3p_9(x)+p_{10}(x)  $$  Now, we have $24x$, $12x^2$, $8x^3$  and $3x^4$ are in $S^{12}$ for all $x\in R$.   So, it remains to show that  $4x^{2m+1}+6x^2+12x$ are in $S^{12}$ for all $x\in R$ and $m\in \mathbb{N}$.  To do this, we calculate under modulo
$$ K_2:=K_1\cup \{3x^{4i},8x^{3i},12x^{2i},24x^i\;|\; x\in R, i\in \mathbb{N}\} \subseteq S^{12}.$$
We observe  that $$p_0(x+1)-p_0(x)=180x^4+780x^3+1326x^2+1044x\equiv_{K_2} 4x^3+6x^2+12x:=g_0(x) \in S^{12},\;\; \forall x\in R.$$
Then $$g_0(x+y)\equiv_{K_2} 12[x^2y+xy^2+xy]:=f(x,y)\in S^{12},\;\; \forall x,y\in R.$$  In particular, when $k\geq 2$, 
$$f(x^2,x^{2k-3}) =12[x^{2k+1}+x^2y^{2(2k-3)}+x^{2k-1}] \equiv_{K_2} 12[x^{2k+1}+x^{2k-1}]:=h(x,k)$$ is in $S^{12}$ for all $x\in R$.  Hence, for each $m\in \mathbb{N}$ such that $m\geq 2$,
$$ 4x^{2m+1}+4x^3 \equiv_{K_2}12[x^{2m+1}+x^3]\equiv_{K_2}h(x,2)+\cdots+h(x,m),$$
which yields that  $$4x^{2m+1}+6x^2+12x \equiv_{K_2}   (4x^{2m+1}+4x^3)+g_0(x),  \in S^{12}\;\; \forall x\in R.$$
Equipped with  $4x^3+6x^2+12x\in S^{12}$, we can now conclude that, for each $m\in \mathbb{N}$,
$4x^{2m+1}+6x^2+12x$ are in $S^{12}$ for all $x\in R$, which completes the proof.

\begin{rem}\label{remarkofdegree12}
	By using $K_2$, $p(t,\delta)$ is reduced to $$f(t,\delta):=12t^{10}\delta+6t^8\delta^2+12t^2\delta^5\in S^{12},$$ for all $t,\delta\in R$.  However, (by trying with so many polynomials $g_1(x),g_2(x)\in R[x]$ ), it seems that $f(g_1(x),g_2(x))\mod K_2$ is in the additive group generated by  $g_0(x)$ and $h(x,k)$, $k\geq 2$.
\end{rem}

\subsection{Proof of Theorem \ref{degprime} for $p=11$}  Due to Theorem \ref{lem39}, the statements (1) implies (2). To show (2) $\longrightarrow$ (1), by Theorem \ref{bthm35}, it suffices to concentrate on matrices in $M_2(R)$.  Let $M\in M_2(R)$.  To show that (1) holds, we use the equivalent statements of (1) and (3) in Theorem \ref{basicthm3.1}.  Let $$S^{11}:=\left< tr(A^{11})\mid A\in M_2(R) \right> \leq R$$ 
be the (additive) subgroup of $R$ generated by traces of $11$-th powers of matrices in $M_2(R)$.  We need to show that $tr(M):=\alpha_0\in S^{11}$.  By the assumption (2),  there exist $\alpha_1, \alpha_2\in R$ such that $$\alpha_0=\alpha_1^{11}+11\alpha_2.$$
Note that $$\alpha_1^{11}=tr(\begin{pmatrix}
\alpha_1 & 0 \\
0 & 0 
\end{pmatrix}^{11})\in S^{11}.$$  Then, in order to show that $\alpha_0\in S$, it suffices to show that $11\alpha\in S$, for each $\alpha\in R$. 

Now, for any $t, \delta\in R$, we have $A(t,\delta):=\begin{pmatrix}
t & \delta \\
-1 & 0 
\end{pmatrix}\in M_2(R)$.  By (\ref{trf}), we compute directly that
\begin{equation}\label{trA11}
tr(A(t,\delta)^{11})=t^{11}-11[t^{9}\delta-4t^{7}\delta^2+7t^{5}\delta^3-5t^3\delta^4+t\delta^5] \in S^{11}, \;\forall t,\delta\in R.
\end{equation}
When $t=1$ and $\delta=0$, we have $1=tr(A(1,0)^{11})\in S$ and thus $\mathbb{Z}\subseteq S^{11}$ (because $S^{11}$ is a group). 
By setting $t=-1, \delta=x$ (in (\ref{trA11})) and $x^{11}\in S^{11}$ equipped with $\mathbb{Z}\subseteq S^{11}$ , we have that 
$$p_0(x):=11[x^5-5x^4+7x^3-4x^2+x] \in S^{11}, \;\forall x\in R.$$
Then $p_1(x):=-[p_0(x)+p_0(-x)]=(2\cdot 11)[5x^4+4x^2]\in S^{11}$ and $p_2(x):=p_0(x)-p_0(-x)]=(2\cdot 11)[x^5+7x^3+x]\in S^{11}$ for all $x\in R$.  So, using $\mathbb{Z}\subseteq S^{11}$, we calculate that 
$$p_3(x):=(p_2(x+1)-p_2(x))-p_1(x)\equiv (2\cdot 11)[10x^3+27x^2+26x]\in S^{11}$$ for all $x\in R$.  This implies that   
$$p_4(x):=p_3(x+1)-p_3(x)\equiv (2\cdot 11)[30x^2+84x]\in S^{11}$$ and
$$p_5(x):=p_3(x)+p_3(-x)= (2\cdot 11)[54x^2]\in S^{11}$$   for all $x\in R$.  So, both $$ p_4(x)-p_4(-x) =4\cdot 11\cdot 84x\;\; \hbox{ and }\;\; p_5(x+1)-p_5(x)\equiv 4\cdot 11\cdot 54x$$  belong to $S^{11}$ for all $x\in R$.  Note that $$\gcd(4\cdot 11\cdot 84x,4\cdot 11\cdot 54x)=24\cdot 11x$$ and that doing Euclidean algorithm for $4\cdot 11\cdot 84x$ and $4\cdot 11\cdot 54x$ produces an element in $S^{11}$.  So, 
$24\cdot 11x \in S^{11}$ for all $x\in R$.   Thus, by using this fact together with $p_2(6x)\in S^{11}$  for all $x\in R$, we conclude that
$$12\cdot 11x \in S^{11},\;\; \forall x\in R.  $$
Since, $p(t,\delta):=tr(A(-t,\delta)^{11})\equiv 11[t^{9}\delta-4t^{7}\delta^2+7t^{5}\delta^3-5t^3\delta^4+t\delta^5]\in S^{11}$, by using $12\cdot 11x\in S^{11}$,  we get that
$$  p(3,2x)\equiv 2\cdot11\cdot 3^9x\in S^{11}, \;\; \forall x\in R.$$ 
Thus, $$6\cdot 11x=\gcd(12\cdot 11x , 2\cdot11\cdot 3^9x)\in S^{11}, \;\; \forall x\in R.  $$
By using $6\cdot 11 x^i\in S^{11}$ for all $x\in R$ and $i\in \mathbb{N}$, $p_0(x)$ is reduced to $$g_0(x):=11[x^5+x^4+x^3+2x^2+x] \in S^{11}, \;\; \forall x\in R, $$
and $p(x,1)$  is reduced to $$h_0(x):=11[x^9+2x^7+x^5+x^3+x] \in S^{11}, \;\; \forall x\in R.$$
Then, 
$$g_1(x):=g_0(x+2)-g_0(x+3)=11[x^4-x^2] $$ and $$g_2(x):=g_0(x+1)-g_0(x)+g_1(x)=11[2x^3+4x]\equiv 11[2x^3-2x]  $$
belong to $S^{11}$ for all $x\in R$.   Since $$K_1:=\mathbb{Z}\cup\{ 6\cdot 11 x^i, 11[x^{4i}-x^{2i}], 11[2x^{3i}-2x^{i}]\;|\; x\in R, i\in \mathbb{N}\}\subseteq S^{11},$$ we compute that
\begin{eqnarray*}
	h_1(x):=h_0(x+1)-h_0(x)&\equiv_{K_1}&11[3x^8+2x^6+3x^4+2x^3+x^2+x] \\
	&\equiv_{K_1}&11[2x^6+2x^3+x^2+x]\\
	&\equiv_{K_1}&11[2x^3+3x^2+x]\\
	&\equiv_{K_1}&11[3x^2+3x]\in S^{11}, \;\; \forall x\in R.
\end{eqnarray*}
We also compute that,$$p(x,2)\equiv_{K_1}11[2x^9+2x^7+2x^5-2x^3+2x] $$
and thus $2h_0(x)-p(x,2)\equiv_{K_1}11[2x^7+4x^3]\equiv_{K_1} 11[2x^7+4x]=:h_2(x)\in S^{11}$ for all $x\in R$.  Now, we have that
\begin{eqnarray*}
	h_3(x):=h_2(x+1)-h_2(x)&\equiv_{K_1}&11[2x^6+2x^4+2x^3+2x]\\
	&\equiv_{K_1}&11[2x^4+2x^3+2x^2+2x]\\
	&\equiv_{K_1}&11[2x^3+4x^2+2x]\\
	&\equiv_{K_1}&11[4x^2+4x]\in S^{11}\;\; \forall x\in R.
\end{eqnarray*}
Then $h_4(x):=h_3(x)-h_1(x)\equiv_{K_1}11[x^2+x]\in S^{11}$ for all $x\in R$ and hence $$ 2\cdot 11x\equiv_{K_1} h_4(x+1)-h_4(x)\in S^{11} \;\;\hbox{ for all } x\in R.$$
Let $K_2:=K_1\cup \{2\cdot 11x^i, 11[x^i+x^{2i}]\;|\; x\in R, i\in \mathbb{N}\}$.  Then $K_2\subseteq S^{11}$ and $$p(t,\delta)\equiv_{K_2} 11[t^9\delta+t^5\delta^3+t^3\delta^4+t\delta^5]:=f(t,\delta).$$
Based on $K_2\subseteq S^{11}$, we compute that:
\begin{eqnarray*}
	f(1,x) &\equiv_{K_2} & 11[x+x^3+x^4+x^5], \\
	f(x,1) &\equiv_{K_2}& 11[x+x^3+x^5+x^9], \\
	f(x+x^3,1) &\equiv_{K_2} & 11[x+x^{13}+x^{15}+x^{25}+x^{27}],\\
	f(1,x+x^3) &\equiv_{K_2} & 11[x+x^{4}+x^{9}+x^{12}+x^{13}+x^{15}]
\end{eqnarray*}
belong to $S$, for all $x\in R$.  This yields
$$f(x+x^3,1)+	f(1,x+x^3)+f(x^3,1)+f(1,x^5)\equiv _{K_2} 11[x^3+x^4+x^5+x^{12}+x^{20}]\in S, \; \forall x\in R.$$ 
By using $11[x^i+x^{2i}]\in K_2$ for all $i\in \mathbb{N}$ and $x\in R$,  the above relation is further reduced to
$$11[x^3+x^4+x^5+x^6+x^{10}]\equiv _{K_2} 11x^4\equiv _{K_2}11x^2\equiv _{K_2}11x,$$
which means that $11x\in S^{11}$ for all $x\in R$.  Hence,  the proof for the equivalence of (1) and (2) is completed.  The assertion for an order in an algebraic number field for $p=11$ in the theorem is an immediate consequence of Theorem \ref{odercri} and the equivalence of  (1) and (2).   The proof of the last sentence follows from the statement (2), Theorem \ref{prop32} and the fact that $\{a_0^{11}+(\operatorname{mod} 11R )\mid a_0\in R\}$ is an additive group.

\subsection{Proof of Theorem \ref{degprime} for $p=13$}  Due to Theorem \ref{lem39}, the statements (1) implies (2). To show (2) $\longrightarrow$ (1), by Theorem \ref{bthm35}, it suffices to concentrate on matrices in $M_2(R)$.  Let $M\in M_2(R)$.  To show that (1) holds, we use the equivalent statements of (1) and (3) in Theorem \ref{basicthm3.1}.  Let $$S^{13}:=\left< tr(A^{13})\mid A\in M_2(R) \right> \leq R$$ 
be the (additive) subgroup of $R$ generated by traces of $13$-th powers of matrices in $M_2(R)$.  We need to show that $tr(M):=\alpha_0\in S^{13}$.  By the assumption (2),  there exist $\alpha_1, \alpha_2\in R$ such that $$\alpha_0=\alpha_1^{3}+13\alpha_2.$$
Note that $$\alpha_1^{13}=tr(\begin{pmatrix}
\alpha_1 & 0 \\
0 & 0 
\end{pmatrix}^{13})\in S^{13}.$$  Then, in order to show that $\alpha_0\in S^{13}$, it suffices to show that $13\alpha\in S^{13}$, for each $\alpha\in R$. 

Now, for any $t, \delta\in R$, we have $A(t,\delta):=\begin{pmatrix}
t & \delta \\
-1 & 0 
\end{pmatrix}\in M_2(R)$.  By (\ref{trf}), we compute directly that
\begin{equation}\label{trA13}
tr(A(t,\delta)^{13})=t^{13}-13[t^{11}\delta-5t^{9}\delta^2+12t^{7}\delta^3-14t^5\delta^4+7t^3\delta^5-t\delta^6] \in S^{13}, \;\forall t,\delta\in R.
\end{equation}
When $t=1$ and $\delta=0$, we have $1=tr(A(1,0)^{13})\in S$ and thus $\mathbb{Z}\subseteq S$ (because $S^{13}$ is a group). 
By setting $t=-1, \delta=x$ (in (\ref{trA13})) and $x^{13}\in S^{13}$ equipped with $\mathbb{Z}\subseteq S^{13}$ , we have that 
$$p(t,\delta):=13[t^{11}\delta-5t^{9}\delta^2+12t^{7}\delta^3-14t^5\delta^4+7t^3\delta^5-t\delta^6] \in S^{13}, \;\forall t,\delta\in R.$$
Then
$$p_0(x):=p(1,x)=13[x-5x^2+12x^3-14x^4+7x^5-x^6] \in S^{13}, \;\forall x\in R.$$
Let $K_0:=\mathbb{Z}\cup \{x^{13i}\;|\; x\in R, i\in \mathbb{N}\}$.  Then $K_0\subseteq S^{13}$.   We now compute that:
\begin{eqnarray*}
	p_1(x):=p_0(x+1)-p_0(x)&\equiv_{K_0}&13[6x^5-20x^4+6x^3-7x^2+x]\\
	p_2(x):=p_0(x)-p_0(-x)&\equiv_{K_0}&13[14x^5+24x^3+2x]\\
	p_3(x):=3p_2(x)-7p_1(x)&\equiv_{K_0}&13[140x^4+30x^3+49x^2-x]\\
	p_4(x):=p_3(x+1)-p_3(x)&\equiv_{K_0}&13[560x^3+930x^2+748x]\\
	p_5(x):=p_3(x)-p_3(-x)&\equiv_{K_0}&13[60x^3-2x]\\
	p_6(x):=3p_4(x)-28p_5(x)&\equiv_{K_0}&13[2790x^2+2300x]\\
	p_7(x):=p_6(x)-p_6(-x)&\equiv_{K_0}&13[4600x]\\
	p_8(x):=p_6(x+1)-p_6(x)&\equiv_{K_0}&13[5580x]
\end{eqnarray*}
belong to $S^{13}$ for all $x\in R$.  Since $\gcd(13[4600x],13[5580x])=13[20x]$ and the Euclidean algorithm for $13[4600x], 13[5580x]$ produces an element in $S^{13}$, we conclude that $13[20x^i]\in S^{13}$ for all $x\in R$ and $i\in \mathbb{N}$.  In particular, we have that $3\cdot 13[20x^3]\in S^{13}$.  By $p_5(x)$, we can also conclude that $13[2x]\in S^{13}$ for all $x\in R$.  Moreover, $p_1(x)$ is reduced to $13[x+x^2]$ and $p_0(x)$ is reduced to $13[x+x^2+x^5+x^6]$ which is also reduced to $13[x^5+x^6]$ (by $13[x+x^2]\in S^{13}$).   By setting $$K_1:=K_0\cup\{13[2x^i], 13[x^i+x^{2i}], 13[x^{5i}+x^{6i}]\;|\;x\in R, i\in \mathbb{N}\},$$ we have $K_1\subseteq S^{13}$ and $p(t,\delta)$ is reduced to
$$f(t,\delta):=13[t^{11}\delta+t^9\delta^2+t^3\delta^5+t\delta^6]\in S^{13},\;\forall t,\delta\in R.  $$
Based on $K_1\subseteq S^{13}$, we compute that:
\begin{eqnarray*}
	f(x,1) &\equiv_{K_1}& 13[x+x^3+x^9+x^{11}], \\
	f(x,1+x^2) &\equiv_{K_1} & 13[x+x^3], \\
	f(x,x+x^3) &\equiv_{K_1}& 13[x^7+x^8+x^{10}+x^{12}+x^{14}+x^{16}+x^{18}+x^{19}], \\
	f(x,1+x^3) &\equiv_{K_1} & 13[x+x^3+x^6+x^7+x^9+x^{11}+x^{14}+x^{18}+x^{19}],
\end{eqnarray*}
belong to $S$, for all $x\in R$.  This yields that $f(x,1)+	f(x,1+x^2)+f(x,x+x^3)+f(x,1+x^3) $ is congruence to
\begin{eqnarray*}
	13[x+x^3+x^6+x^8+x^{10}+x^{12}+x^{16}] 
	&\equiv _{K_1}& 13[x+(x^3+x^6)+(x^8+x^{16})+(x^{10}+x^{12})]\\
	&\equiv _{K_1}& 13x\in S^{13}
\end{eqnarray*} 
for all $x\in R$.  Hence,  the proof for the equivalence of (1) and (2) is completed.  The assertion for an order in an algebraic number field for $p=13$ in the theorem is an immediate consequence of Theorem \ref{odercri} and the equivalence of  (1) and (2).   The proof of the last sentence follows from the statement (2), Theorem \ref{prop32} and the fact that $\{a_0^{13}+(\operatorname{mod} 13R )\mid a_0\in R\}$ is an additive group.

\subsection{Proof of Theorem \ref{deg16}}  Due to Theorem \ref{lem39}, each of the statements (1), (2), (3), (4) implies (5).  It is also obvious that (4) $\longrightarrow$ (3) $\longrightarrow$ (2) $\longrightarrow$ (1).  So, it remains to show only that (5) $\longrightarrow$ (4) for $M_2(R)$ (by Theorem \ref{bthm35}). Suppose that (5) holds.  Let $M\in M_2(R)$.  To show that (4) holds, we use the equivalent statements of (1) and (3) in Theorem \ref{basicthm3.1}.  Let $$S^{16}:=\left< tr(A^{16})\mid A\in M_2(R) \right> \leq R$$ 
be the (additive) subgroup of $R$ generated by traces of $16$-th powers of matrices in $M_2(R)$.  We need to show that $tr(M):=\alpha_0\in S^{16}$.  By the assumption (5) and the same proof of Proposition \ref{prop1}, for any positive integer $m$, there exist $\alpha_1, \alpha_2,\dots,\alpha_m,\alpha_m' \in R$ such that $$\alpha_0=\alpha_1^{16}+2\alpha_2^{16}+\cdots +2^{m-1}\alpha_m^{16}+2^m \alpha_m'.  $$
Note that $$\alpha_i^{16}=tr(\begin{pmatrix}
\alpha_i & 0 \\
0 & 0 
\end{pmatrix}^{16})\in S^{16}$$ for all $i=1,\dots,m$.  Then, in order to show that $\alpha_0\in S^{16}$, it suffices to show that, for each $\alpha\in R$, there exists a positive integer $m$ such that $2^m\alpha\in S^{16}$.

Now, for any $t, \delta\in R$, we have $A(t,\delta):=\begin{pmatrix}
t & \delta \\
-1 & 0 
\end{pmatrix}\in M_2(R)$.  By (\ref{trf}), we compute directly that
\begin{equation}\label{trA16}
tr(A(t,\delta)^{16})=t^{16}-16t^{14}\delta+104t^{12}\delta^2-352t^{10}\delta^3+660t^8\delta^4-672t^6\delta^5+336t^4\delta^6-64t^2\delta^7+2\delta^8.
\end{equation}
When $t=1$ and $\delta=0$, we have $1=tr(A(1,0)^{16})\in S^{16}$ and thus $\mathbb{Z}\subseteq S$ (because $S^{16}$ is a group).  By setting $t=0, \delta=x$ in (\ref{trA16}), we have that $$P_0(x):=2x^8=tr(A(0,x)^{16})\in S^{16}, \; \forall x\in R.$$ 
By setting $t=1, \delta=x$ (in (\ref{trA16})) equipped with $\mathbb{Z}\subseteq S^{16}$ and $P_0(x)\in S^{16}$ for all $x\in R$, we have that 
\begin{equation}\label{Q016}
Q_0(x):=-16x+104x^2-352x^3+660x^4-672x^5+336x^6-64x^7\in S^{16}, \;\forall x\in R.
\end{equation}
Then $P_0(x+1)-P_0(x-1)\in S^{16}$ for all $x\in R$, and,  by using  $\mathbb{Z}\subseteq S$ and $P_0(x)\in S^{16}$, we have that
$$P_1(x):=112x^6+280 x^4+112 x^2 \in S^{16}, \; \forall x\in R. $$
Also, $Q_0(x)+Q_0(-x)\in S^{16}$ for all $x\in R$, and, by using  $\mathbb{Z}\subseteq S^{16}$, we have that 
$$Q_1(x):=672x^6+1320 x^4+208 x^2 \in S^{16}, \; \forall x\in R. $$
This implies that $6P_1(x)-Q_1(x)\in S^{16}$ for all $x\in R$; namely
$$P_2(x):=360 x^4+464 x^2 \in S^{16}, \; \forall x\in R.  $$
Moreover, $Q_1'(x):=Q_1(x+1)-Q_1(x-1)\in S^{16}$ and thus $Q_1'(x+1)-Q_1'(x-1)\in S^{16}$ for all $x\in R$.  After using $\mathbb{Z}\subseteq S^{16} $, we get that 
$$Q_2(x):=80,640 x^4+385,920 x^2 \in S^{16}, \; \forall x\in R,  $$
and thus $Q_2(x)-224 P_2(x)\in S^{16}$ for all $x\in R$; namely
$$Q_3(x)=281,984 x^2\in S^{16},\; \forall x\in R. $$
Furthermore, $P_2'(x):=P_2(x+1)-P_2(x-1)\in S^{16}$ and thus $P'_2(x+1)-P'_2(x-1)\in S^{16}$ for all $x\in R$.  After using $\mathbb{Z}\subseteq S^{16}$, we obtain that
$$ P_3(x):=17,280 x^2\in S^{16}, \; \forall x\in R. $$
Note that $\gcd(P_3(x),Q_3(x))=128x^2$ and that doing Euclidean algorithm for $P_3(x)$ and $Q_3(x)$ produces an element in $S^{16}$. We now reach to the fact that
$$R(x):=128 x^2 \in S^{16},\; \forall x\in R. $$
Therefore, $R(x+1)\in S^{16}$ for all $x\in R$ and after using $\mathbb{Z}\subseteq S^{16}$ and $R(x)\in S$, we finally conclude that
$$2^8 x=256 x \in S^{16},\; \forall x\in R,$$
which completes the proof for the equivalence of (4) and (5).  The last assertion in the theorem is an immediate consequence of Theorem \ref{odercri} and the equivalence of one of (1)-(4) and (5).

To prove the statement (6), we consider $tr(M)\in W^*(2,4,R)$; namely $$tr(M)=a_0^{16}+2a_1^{8}+4a^4+16a+16a^{2m+1}$$  or $$tr(M)=b_0^{32}+2b_1^{16}+4b_2^{8}+8b_3^{4}+16b_4^2+32 b_5$$ for some $a_0, a_1, a, b_0, b_1,b_2,b_3,b_4,b_5\in R$ and $m\in \mathbb{N}$. To complete the proof, it suffices to show that  $\alpha^{16}$,  $2\alpha^8$, $8\alpha^4$, $16\alpha^2$, $32\alpha$ and $4\alpha^4+16\alpha+16\alpha^{2m+1}$ for any $m\in \mathbb{N}$ are sums of traces of sixteenths powers of matrices in $M_2(R)$, for each $\alpha\in R$; the sum of sixteenth powers matrices in the form (direct sum) $N\oplus O_{n-2}$, with $N\in M_2(R)$ and zero matrix $O_{n-2}\in M_{n-2}(R)$, will full fill the required sum of sixteenth powers matrices in $M_n(R)$. 
  
  Again, we apply the equivalent statements of (1) and (3) in Theorem \ref{basicthm3.1} and also use the fact that $$\alpha^{16}, \quad P_0(\alpha):=2\alpha^8, \quad R(\alpha):=128\alpha^2, \quad P_2(\alpha):= 360\alpha^4+464\alpha^2  $$
belong to the additive group $S^{16}$, for each $\alpha\in R$. So, to finish the proof, it remains to show that $8\alpha^4$, $16\alpha^2$, $32\alpha$ and $4\alpha^4+16\alpha+16x^{2m-1}$ are in $S^{16}$, for each $\alpha\in R$ and $m\in \mathbb{N}$.  Now, by using $R(\alpha)\in S^{16}$ and $P_2(\alpha)\in S^{16}$, we get that $R_1(\alpha):=[3R(\alpha^2)+4R(\alpha)]-P_2(\alpha)=24\alpha^4+48\alpha^2\in S^{16}$ for all $\alpha\in R$. This yields that \begin{equation*}
R_2(\alpha):=R(\alpha^2)-5R_1(\alpha)+2R(\alpha)=8\alpha^4+16\alpha^2 \in S^{16},\;\; \forall \alpha\in R,
\end{equation*}
and that 
\begin{equation*}
R_3(\alpha):=R_2(2\alpha)-2[8R_2(\alpha)-R(\alpha)]=64\alpha^2 \in S^{16},\;\; \forall \alpha\in R.
\end{equation*}
So, by using $\mathbb{Z}\subseteq S^{16}$ and $R_3(\alpha+1)\in S$, we obtain that $R_3(\alpha+1)-R_3(\alpha)-64=128\alpha $ is in $S^{16}$ for all $\alpha\in R$. Then, $L:=\mathbb{Z}\cup\{128\alpha^i, 64\alpha^{2i},2\alpha^{8i}, \alpha^{16i}\mid \alpha\in R, i\in \mathbb{N}\}\subseteq S^{16}$.  By (\ref{Q016}), $Q_0(2\alpha)\in S$ and thus (by using $L$), $$R_4(\alpha):=-32\alpha+32\alpha^2 \equiv_L Q_0(2\alpha)\in S, \;\; \forall \alpha\in S^{16}.$$
 Then $$R_5(\alpha):=4R_2(\alpha)-R_4(\alpha^2)=96 \alpha^2\in S^{16},\;\; \forall \alpha\in R,  $$
 and thus we have proved that $$32\alpha=R_5(\alpha)-R_4(\alpha)-R_3(\alpha)\in S^{16},\;\; \forall \alpha\in R.$$
It also implies that, $T:=L\cup\{32\alpha^i\mid \alpha\in R,  i\in \mathbb{N}\}\cup \mathbb{Z}\subseteq S$.  By using this fact, we compute that $$ 16\alpha^2\equiv_T R_2(\alpha+1)-R_2(\alpha) \in S^{16}, \;\; \forall \alpha\in R, $$ 
and hence, by $R_2(\alpha)$ again, $8\alpha^4\in S^{16}$ for all $\alpha\in R$. So far, $$U:=T\cup\{16\alpha^{2i}, 8\alpha^{4i}\mid \alpha\in R, i\in \mathbb{N}\}\subseteq S^{16}$$ and therefore, by $Q_0(\alpha)$ in (\ref{Q016}), we conclude that
$$q_0(\alpha):=4\alpha^4+8\alpha^2+16\alpha\equiv_U Q_0(\alpha)\in S^{16},\;\; \forall \alpha\in R.  $$
Note further that $q_0(\alpha+1)\equiv_U 16\alpha^3+8\alpha^2:=g(\alpha)\in S^{16} $ for all $\alpha\in R$. Then $$q_0(\alpha)+g(\alpha)\equiv_U 4\alpha^4+16\alpha^3+16\alpha =: q(\alpha)\in S^{16},\;\; \forall \alpha\in R. $$
Moreover, $$ g(\alpha+\beta)\equiv_U 16[\alpha^2\beta+\alpha\beta^2+\alpha\beta]=:f(\alpha,\beta)\in S^{16},\;\; \forall \alpha,\beta\in R.$$
In particular, when $k\geq 2$, 
$$f(\alpha^2,\alpha^{2k-3}) =16[\alpha^{2k+1}+\alpha^2\alpha^{2(2k-3)}+\alpha^{2k-1}] \equiv_{U} 16[\alpha^{2k+1}+\alpha^{2k-1}]:=h(\alpha,k)$$ is in $S^{16}$ for all $\alpha\in R$.  Hence, for each $m\in \mathbb{N}$ such that $m\geq 2$,
$$ 16\alpha^{2m-1}+16\alpha^3 =16[\alpha^{2m-1}+\alpha^3]\equiv_{U}h(\alpha,2)+\cdots+h(\alpha,m)\in S^{16},$$
which yields that  $$4\alpha^4+16x^{2m-1}+16\alpha \equiv_{U}   (16\alpha^{2m-1}+16x^3)+l_1(x),  \in S^{16}\;\; \forall \alpha\in R.$$
Equipped with  $4\alpha^4+16\alpha^3+16\alpha\in S^{16}$, we can now conclude that, for each $m\in \mathbb{N}$,
$4\alpha^4+16\alpha+16x^{2m+1}$ are in $S^{16}$ for all $\alpha\in R$, which completes the proof.

\begin{rem}\label{remarkofdegree16}
	By using $U$, $tr(A(t,\delta)^{16})$ is reduced to $$F(t,\delta):=16t^{14}\delta+8t^{12}\delta^2+4t^8\delta^4\in S^{16},$$ for all $t,\delta\in R$.  However, (by trying with so many polynomials $g_1(\alpha),g_2(\alpha)\in R[\alpha]$ ), it seems that $F(g_1(\alpha),g_2(\alpha))\mod U$ is in the additive group generated by  $g(\alpha),q(\alpha)$ and $h(\alpha,m)$, $m\geq 2$.
\end{rem}

\section*{Acknowledgments}
The second author would like to thank Faculty of Science, Naresuan University, for financial support on the project number P2565C007.

\end{document}